\documentclass[11pt]{amsart}
\usepackage{geometry}
\usepackage[mathscr]{euscript}
\usepackage{ amssymb }
\usepackage[comma, numbers]{natbib}
\usepackage{bm}
\usepackage{enumerate}
\usepackage[english]{babel}
\usepackage{commath}
\usepackage{tikz-cd}
\usepackage{graphicx}
\usepackage [autostyle, english = american]{csquotes}
\MakeOuterQuote{"}

\theoremstyle{plain}
\newtheorem{theorem}{Theorem}[section]

\theoremstyle{plain}
\newtheorem{lemma}[theorem]{Lemma}

\theoremstyle{plain}
\newtheorem{corollary}[theorem]{Corollary}

\theoremstyle{definition}
\newtheorem{definition}[theorem]{Definition}

\theoremstyle{plain}
\newtheorem{proposition}[theorem]{Proposition}

\theoremstyle{remark}
\newtheorem{remark}[theorem]{Remark}

\theoremstyle{definition}
\newtheorem{example}[theorem]{Example}

\theoremstyle{plain}

\theoremstyle{plain}

\theoremstyle{plain}
\newtheorem{question}[theorem]{Question}

\title[an alternative description of coarse proximities]{Normality conditions of structures in coarse geometry and an alternative description of coarse proximities}

\author{Pawel Grzegrzolka}
\address{University of Tennessee, Knoxville, USA}
\email{pgrzegrz@vols.utk.edu}

\author{Jeremy Siegert}
\address{University of Tennessee, Knoxville, USA} 
\email{jsiegert@vols.utk.edu}
\date{\today} 

\keywords{coarse geometry, coarse topology, coarse proximity, proximity, large scale normality, coarse normality, alternative definition of coarse proximities, coarse spaces, asymptotic resemblance spaces\\
published in \textit{Topology Proceedings}, 53:285-299, 2019}
\subjclass[2010]{54E05, 54E15, 51F99}

\begin{document}

\begin{abstract}
We introduce an alternative description of coarse proximities. We define a coarse normality condition for connected coarse spaces and show that this definition agrees with large scale normality defined in \cite{DydakandWeighill} and asymptotic normality defined in \cite{Honari}. We utilize the alternative definition of coarse proximities to show that a connected coarse space naturally induces a coarse proximity if and only if the connected coarse space is coarsely normal. We conclude with showing that every connected asymptotic resemblance space induces a coarse proximity if and only if the connected asymptotic resemblance space is asymptotically normal.
\end{abstract}

\maketitle
\tableofcontents

\section{Introduction}
Coarse topology (i.e., large-scale geometry) studies large-scale properties of spaces (e.g., asymptotic dimension, property A, exactness). It emerged as a counterpart to classical topology, which is usually concerned with small-scale properties of spaces (e.g., continuity, compactness). Tools and techniques developed by coarse topologists are often useful in other branches of mathematics, including geometric group theory (see \cite{Guentner}), index theory (see \cite{indextheory}), and dimension theory (see \cite{dimensiontheory}). Coarse topology is also closely related to well-known conjectures, including the Novikov conjecture (see \cite{Novikov}) and the coarse Baum-Connes conjecture (see \cite{Baum-Connes} or \cite{Baum}).

Recently in \cite{Coarseproximity} coarse proximities were introduced as a large-scale counterpart to (small-scale) proximities. Coarse proximities capture the intuitive notion of two sets being "close" at infinity. The motivation for these structures comes in part from the utility of small-scale proximity relations in classical topology. More specifically, compatible proximity relations on a Tychonoff topological space classify the compactifications of that space. One can hope to port some of this utility to the large-scale context with coarse proximities. Another reason to work on coarse proximity spaces, aside from the linguistic one just mentioned, is the presence of proximity ideas in extant coarse geometry literature. In \cite{asymptoticinductivedimension}, Dranishnikov introduced the notion of subsets of a metric space being asymptotically disjoint, and in \cite{asymptotictopology} he introduced the notion of asymptotic neighbourhoods of subsets of metric spaces. In \cite{asymptoticinductivedimension}, the relation of asymptotic disjointness was used to define a dimensional coarse invariant of proper metric spaces, the asymptotic inductive dimension. It has been shown that for proper metric spaces of positive but finite asymptotic dimension the asymptotic inductive dimension and the asymptotic dimension agree (see \cite{asymptoticinductivedimension}). In \cite{bellanddranishnikov}, Bell and Dranishnikov used the notion of asymptotic disjointness to define another coarse dimensional invariant, the asymptotic Brouwer inductive dimension. In \cite{Coarseproximity}, it was shown that the relation of being asymptotically disjoint is equivalent to the negation of the metric coarse proximity relation. Likewise, in \cite{Coarseproximity} it was shown that the asymptotic neighbourhoods of Dranishnikov are equivalent to coarse neighbourhoods in the metric coarse proximity structure. Aside from its utility in defining the aforementioned coarse invariants, coarse proximity notions were used in \cite{Coarseproximity} to define the "proximity space at infinity:" a (small-scale) proximity space that corresponds to the coarse proximity structure induced by an unbounded metric space. The construction of the proximity space at infinity extensively utilizes the coarse neighbourhoods of the metric coarse proximity structure. The proximity space at infinity is shown to comprise a functor from the category of unbounded metric spaces whose morphisms are closeness classes of coarse proximity maps (or coarse maps) to the category of proximity spaces whose morphisms are proximity maps. Consequently, the proximity isomorphism type of the proximity at infinity of an unbounded metric space $X$ is a coarse invariant of $X$. 

The focus of this paper is to characterize part of the relationship coarse spaces (for an introduction to coarse spaces, see \cite{Roe}) have with coarse proximities. In the small-scale context, every uniform structure (defined via entourages or uniform covers) induces a proximity relation that is compatible with the topology of the original uniform structure. It is natural to ask whether or not the large-scale analog of uniform structures, coarse spaces, have a similar relationship with coarse proximities. To aid in answering this question, we provide an alternative characterization of coarse proximity spaces in terms of coarse neighbourhoods in section \ref{Alternative_Definition_of_Coarse_Proximity}. In section \ref{large_scale_normality_section}, coarse normality of coarse spaces is introduced. We show that our definition of coarse normality is equivalent to the large scale normality of Dydak and Weighill as introduced in \cite{DydakandWeighill}. We also show that a coarse space being coarsely normal is equivalent to its induced asymptotic resemblance structure being asymptotically normal, as defined in \cite{Honari}. Our alternative characterization of coarse proximities reveals its utility in section \ref{inducing coarse proximities}, where it is used to answer our original question. Specifically, it is used to show that a coarse space naturally induces a coarse proximity if and only if the coarse space is coarsely connected and coarsely normal. By "natural" we mean that the proximity space induced by that coarse structure agrees with the metric coarse proximity when the underlying space is a metric space and the coarse structure is induced by that metric. In other words, this result generalizes the result from \cite{Coarseproximity}, where it is shown that every metric space induces a coarse proximity structure. As a corollary, we also obtain that an asymptotic resemblance (for an introduction to asymptotic resemblances, see \cite{Honari}) induces a coarse proximity if and only if it is asymptotically connected and asymptotically normal. We conclude with three equivalent characterizations of the coarse proximity induced by a coarsely connected and coarsely normal coarse space.

\section{An Alternative Definition of Coarse Proximities}\label{Alternative_Definition_of_Coarse_Proximity}

In this section, we introduce a definition of a coarse proximity in terms of coarse neighborhoods. Recall the following three definitions from \cite{Coarseproximity}:

\begin{definition}
	A {\bf bornology} $\mathcal{B}$ on a set $X$ is a family of subsets of $X$ satisfying:
	\begin{enumerate}[(i)]
		\item $\{x\}\in\mathcal{B}$ for all $x\in X,$
		\item $A\in\mathcal{B}$ and $B\subseteq A$ implies $B\in\mathcal{B},$
		\item If $A,B\in\mathcal{B},$ then $A\cup B\in\mathcal{B}.$
	\end{enumerate}
Elements of $\mathcal{B}$ are called {\bf bounded} and subsets of $X$ not in $\mathcal{B}$ are called {\bf unbounded}.
\end{definition}

\begin{definition}\label{coarseproximitydefinition}
Let $X$ be a set equipped with a bornology $\mathcal{B}$. A \textbf{coarse proximity} on a set $X$ is a relation ${\bf b}$ on the power set of $X$ satisfying the following axioms for all $A,B,C \subseteq X:$

\begin{enumerate}[(i)]
	\item $A{\bf b}B$ implies $B{\bf b}A,$ \label{axiom1}
	\item $A{\bf b}B$ implies $A \notin \mathcal{B}$ and $B \notin \mathcal{B},$ \label{axiom2}
	\item $A\cap B \notin \mathcal{B}$ implies $A {\bf b} B,$ \label{axiom3}
	\item $(A \cup B){\bf b}C$ if and only if $A{\bf b}C$ or $B{\bf b}C,$ \label{axiom4}
	\item $A\bar{\bf b}B$ implies that there exists a subset $E$ such that $A\bar{\bf b}E$ and $(X\setminus E)\bar{\bf b}B,$ \label{axiom5}
\end{enumerate}
where $A\bar{ {\bf b}}B$ means "$A{\bf b} B$ is not true." If $A {\bf b} B$, then we say that $A$ is \textbf{coarsely close} to (or \textbf{coarsely near}) $B.$ Axiom (\ref{axiom4}) will be called the \textbf{union axiom} and axiom (\ref{axiom5}) will be called the \textbf{strong axiom}. A triple $(X,\mathcal{B},{\bf b})$ where $X$ is a set, $\mathcal{B}$ is a bornology on $X$, and ${\bf b}$ is a coarse proximity relation on $X,$ is called a {\bf coarse proximity space}.
\end{definition}

\begin{definition}\label{neighborhood_definition}
	Let $(X,\mathcal{B},{\bf b})$ be a coarse proximity space. Given subsets $A,B\subseteq X,$ we say that $B$ is a {\bf b-coarse neighborhood} (or just \textbf{coarse neighborhood} if the proximity relation is clear) of $A,$ denoted $A\ll B,$ if $A\bar{\bf b}(X\setminus B)$. 
\end{definition}

\begin{theorem}\label{propertiesofcoarseneighborhoods}
Given a coarse proximity space $(X,\mathcal{B},{\bf b}),$ the relation $\ll $ satisfies the following properties:
\begin{enumerate}
\item $X\ll (X \setminus D)$ for all $D \in \mathcal{B},$ \label{boundeddontmatter}
\item $A \ll B$ implies that $A \subseteq B$ up to some bounded set $D,$ i.e., there exists $D \in \mathcal{B}$ such that $A\setminus D \subseteq B,$ \label{containment}
\item $A \subseteq B \ll C \subseteq D$ implies $A\ll D,$ \label{subsetandsuperset}
\item $A \ll B_1$ and $A \ll B_2$ if and only if $A \ll (B_1 \cap B_2),$ \label{intersection}
\item $A\ll B$ if and only if $(X \setminus B)\ll (X \setminus A),$ \label{neighborhoodofcomplement}
\item $A\ll B$ implies that there exists $C \subseteq X$ such that $A\ll C \ll B.$ \label{intermediate}
\end{enumerate}
\end{theorem}

\begin{proof}
Axiom (\ref{axiom2}) of a proximity space implies that bounded sets are not related to any sets. Thus, $X\bar{\bf b} D$ for any $D \in \mathcal{B}.$ This is the same as saying $X\bar{\bf b} (X \setminus (X \setminus D))$ for any $D \in \mathcal{B},$ or equivalently $X\ll (X \setminus D)$ for any $D \in \mathcal{B},$ which is the statement of (\ref{boundeddontmatter}). To show (\ref{containment}), notice that if $A \cap (X \setminus B) \notin \mathcal{B},$ then $A{\bf b} (X \setminus B),$ a contradiction to $A\ll B.$ To show (\ref{subsetandsuperset}), for contradiction assume that $A\not\ll D,$ i.e., $A{\bf b}(X\setminus D).$ The union axiom implies then that $B{\bf b}(X\setminus D).$ Since $(X\setminus D) \subseteq (X\setminus C),$ again by the union axiom we get $B{\bf b}(X\setminus C),$ a contradiction to $B \ll C.$
To show (\ref{intersection}), notice that by the union axiom
\begin{equation*}
\begin{split}
A \ll B_1 \text { and } A \ll B_2 & \Longleftrightarrow A\bar{\bf b} (X \setminus B_1) \text{ and } A\bar{\bf b} (X \setminus B_2)\\
& \Longleftrightarrow A\bar{\bf b} ((X \setminus B_1) \cup (X \setminus B_2))\\
& \Longleftrightarrow A\bar{\bf b} (X \setminus (B_1 \cap B_2))\\
& \Longleftrightarrow A \ll (B_1 \cap B_2).
\end{split}
\end{equation*}
 To show (\ref{neighborhoodofcomplement}), notice that 
\begin{equation*}
\begin{split}
A\ll B & \Longleftrightarrow A\bar{\bf b} (X \setminus B)\\
& \Longleftrightarrow (X \setminus B)\bar{\bf b}A\\
& \Longleftrightarrow (X \setminus B)\bar{\bf b} (X \setminus (X \setminus A))\\
& \Longleftrightarrow (X \setminus B)\ll (X \setminus A).
\end{split}
\end{equation*}
To show (\ref{intermediate}), assume $A\ll B,$ i.e., $A \bar{\bf b} (X \setminus B).$ The strong axiom implies that there exists $E \subseteq X$ such that $A \bar{\bf b} E$ and $(X \setminus E) \bar{\bf b} (X \setminus B).$ In other words, we have that $A \bar{\bf b} (X \setminus (X \setminus E))$ and  $(X \setminus E) \bar{\bf b} (X \setminus B),$ i.e., $A\ll (X \setminus E) \ll B.$ Setting $C=(X \setminus E)$ gives the desired result.
\end{proof}

\begin{theorem}\label{<<inducesproximity}
Let $X$ be a set with bornology $\mathcal{B}.$ Let $\ll $ be a binary relation on the power set of $X$ satisfying $(1)$ through $(6)$ of Theorem \ref{propertiesofcoarseneighborhoods}. Let ${\bf b}$ be a relation on the power set of $X$ defined by
\[A\bar{{\bf b}} B \quad \text{if and only if} \quad A\ll (X \setminus B).\]
Then ${\bf b}$ is a coarse proximity on $X.$ Also, $B$ is a ${\bf b}$-coarse neighborhood of $A$ if and only if $A\ll B.$
\end{theorem}

\begin{proof}
To show axiom (\ref{axiom1}), assume $A \bar{\bf b} B$. Then $A\ll (X \setminus B),$ which by (\ref{neighborhoodofcomplement}) implies that  $B\ll (X \setminus A),$ i.e., $B \bar{\bf b} A.$ To show axiom (\ref{axiom2}), notice that (\ref{boundeddontmatter}) \and (\ref{subsetandsuperset}) imply that $A \ll (X\setminus B)$ for all $B \in \mathcal{B},$ i.e., $A \bar{\bf b} B$ for all $B \in \mathcal{B}.$ By symmetry proven in axiom (\ref{axiom1}), this implies axiom (\ref{axiom2}). To show axiom (\ref{axiom3}), assume $A \bar{\bf b} B,$ i.e., $A\ll (X \setminus B).$ By (\ref{containment}), this means that there exists $D \in \mathcal{B}$ such that $(A\setminus D) \subseteq (X \setminus B),$ which is the the same as saying that $(A\setminus D) \cap B = \emptyset.$ Thus, $A \cap B\subseteq D,$ showing that $A \cap B \in \mathcal{B}.$ To show axiom (\ref{axiom4}), first assume $(A\cup B) \bar{\bf b} C,$ i.e., $(A \cup B) \ll (X \setminus C).$ Property (\ref{subsetandsuperset}) implies that $A \ll (X \setminus C)$ and $B \ll (X \setminus C),$ i.e., $A\bar{\bf b} C$ and $B\bar{\bf b} C.$ To prove the forward direction, assume $(A\cup B){\bf b} C,$ which by symmetry gives us $C{\bf b} (A\cup B),$ i.e., $C\not\ll X\setminus (A\cup B).$ This is the same as saying $C\not\ll ((X\setminus A)\cap (X \setminus B))$ which by (\ref{intersection}) implies that $C\not\ll (X\setminus A)$ or $C\not\ll (X \setminus B),$ i.e. $C {\bf b} A$ or $C {\bf b} B.$ This again by symmetry implies that $A {\bf b} C$ or $B {\bf b} C,$ proving (\ref{axiom4}). To show the strong axiom, assume $A\bar{\bf b}B,$ i.e., $A\ll (X \setminus B).$ Therefore, by (\ref{intermediate}), there exists $C \subseteq X$ such that $A\ll C\ll (X \setminus B),$ or equivalently$A\ll (X \setminus (X \setminus C))\ll (X \setminus B).$ This implies that $A \bar{\bf b} (X \setminus C)$ and $C \bar{\bf b} B.$ Let $E=X \setminus C.$ Then $A \bar{\bf b} E$ and $(X\setminus E) \bar{\bf b} B.$ Finally, notice that
\begin{equation*}
\begin{split}
B \text{ is a ${\bf b}$-coarse neighborhood of }A & \Longleftrightarrow A \bar{\bf b} (X \setminus B)\\
& \Longleftrightarrow A\ll (X \setminus (X\setminus B))\\
& \Longleftrightarrow A\ll B. \qedhere
\end{split}
\end{equation*}
\end{proof}

The reader is encouraged to compare the above theorems with the similar theorems for small-scale proximity spaces (see Theorem $3.9$ and Theorem $3.11$ in \cite{proximityspaces}). 
Recall that a (small-scale) \textbf{proximity space} is a pair $(X,\delta),$ where $X$ is a set and $\delta$ is a relation on the power set of $X$ that satisfies all the axioms of Definition \ref{coarseproximitydefinition} with "$ \notin \mathcal{B}$" replaced by "$\neq \emptyset$". A $\delta$-neighborhood, with $B$ being a $\delta$-neighbourhood of $A$ being denoted by $A\ll B$, is defined as in Definition \ref{neighborhood_definition} with $\bar{\bf b}$ replaced by $\bar{\delta}.$ 
The above two theorems show the utility of the similarity of definitions of coarse proximities and proximities. For example, the proofs of properties (\ref{subsetandsuperset}) through (\ref{intermediate}) of Theorem \ref{propertiesofcoarseneighborhoods} only use axioms (\ref{axiom1}),(\ref{axiom4}), and (\ref{axiom5}) of coarse proximities. Since these axioms are exactly the same for small-scale proximities, small-scale proximities also satisfy properties (\ref{subsetandsuperset}) through (\ref{intermediate}) of Theorem \ref{propertiesofcoarseneighborhoods} (with $\delta$-neighborhoods replacing coarse neighborhoods).

\begin{definition}
In the setting of the above theorem, we say that the relation $\ll$ \textbf{induces a coarse proximity on the pair} $(X, \mathcal{B}).$
\end{definition}

\begin{remark}
One can show (see \cite{Coarseproximity}) that assuming (\ref{axiom1}) through (\ref{axiom4}) of Definition \ref{coarseproximitydefinition}, the strong axiom of a proximity space is equivalent to the property (\ref{intermediate}) of Theorem \ref{propertiesofcoarseneighborhoods}.
\end{remark}

\section{Coarse Normality}\label{large_scale_normality_section}

In this section, we introduce coarse normality of coarse spaces. We also show that for connected coarse spaces it agrees with large-scale normality introduced by Dydak and Weighill in \cite{DydakandWeighill} and  with asymptotic normality introduced by Honari and Kalantari in \cite{Honari}.

Let us first recall basic definitions related to coarse spaces (from \cite{Roe}) and asymptotic resemblance spaces (from \cite{Honari}). An experienced reader may want to skip ahead to Definition \ref{importantdefinition} and refer to the beginning of this section when necessary. The following $3$ definitions and an example come from \cite{Roe}:

\begin{definition}
A \textbf{coarse structure} on a set $X$ is is a collection $\mathcal{E}$ of subsets of $X \times X,$ called \textbf{controlled sets} or \textbf{entourages}, such that the following are satisfied:
\begin{enumerate}[(i)]
\item $\triangle \in \mathcal{E}$, where $\triangle:=\{(x, x) \mid x \in X\},$
\item if $E \in \mathcal{E}$ and $B \subseteq E,$ then $B \in \mathcal{E},$
\item if $E \in \mathcal{E},$ then $E^{-1} \in \mathcal{E},$ where $E^{-1}:=\{(x,y) \mid (y,x) \in E\},$
\item if $E \in \mathcal{E}$ and $F \in \mathcal{E},$ then $E \cup F \in \mathcal{E},$
\item if $E \in \mathcal{E}$ and $F \in \mathcal{E},$ then $E \circ F \in \mathcal{E},$ where $E \circ F:=\{(x,y) \mid \exists \, z \in X \text{ such that } (x,z) \in E, (z,y) \in F\}.$
\end{enumerate}
A set $X$ endowed with a coarse structure $\mathcal{E}$ is called a \textbf{coarse space}.
\end{definition}

\begin{example}
Let $(X,d)$ be a metric space. For each $r \in \mathbb{R}^+,$ define
\[E_r=\{(x,y) \in X \times X \mid d(x,y)<r\}.\]
Let $\mathcal{E}$ be the collection of all the subsets of such sets $E_r.$ Then $\mathcal{E}$ is a coarse structure, called \textbf{metric coarse structure.}
\end{example}

For more examples of coarse spaces, the reader is referred to \cite{Roe}.

\begin{definition}
If $(X, \mathcal{E})$ is a coarse space, $A$ a subset of $X,$ and $E$ a controlled set, then we define
\[E[A]=\{x \in X \mid \exists \, a \in A \text { such that } (x,a) \in E\}.\]
\end{definition}

\begin{definition}
If $(X, \mathcal{E})$ is a coarse space and $A$ a subset of $X,$ then we say that $A$ is (coarsely) \textbf{bounded} if $A \times A$ is a controlled set. If $A$ is not bounded, then we say that $A$ is (coarsely) \textbf{unbounded}.
\end{definition}

\begin{proposition}
If $(X, \mathcal{E})$ is a (coarsely) \textbf{connected coarse space} (i.e., each point of $X \times X$ belongs to some controlled set), then the collection of bounded sets forms a bornology on $X,$ which we call the \textbf{bornology induced by $\mathcal{E}$}.
\end{proposition}
\begin{proof}
Straightforward. 
\end{proof}

The following definitions and examples (up to Defintion \ref{definition6123}) come from \cite{Honari}:

\begin{definition}\label{asymptoticresemblance}
Let $X$ be a set. Let $\lambda$ be an equivalence relation on the power set of $X.$ Then $\lambda$ is called an \textbf{asymptotic resemblance} if it satisfies the following properties:
\begin{enumerate}[(i)]
\item $A_1\lambda B_1,\,A_2\lambda B_2\text{ implies }(A_1\cup A_2)\lambda(B_1\cup B_2),$ \label{condition1}
\item $(B_{1}\cup B_{2})\lambda A$ and $B_{1},B_{2}\neq\emptyset$ implies that there are nonempty $A_{1},A_{2}\subseteq A$ such that $A=A_{1}\cup A_{2}$, $B_{1}\lambda A_{1}$, and $B_{2}\lambda A_{2}$. \label{condition2}
\end{enumerate}
A pair $(X, \lambda),$ where $X$ is a set and $\lambda$ is an asymptotic resemblance, is called an \textbf{asymptotic resemblance space}.
\end{definition}

\begin{example}
Let $(X,d)$ be a metric space, and $A,C \subseteq X.$ Let $B(A,r)$ denote the neighborhood of radius $r$ around $A,$ i.e., $B(A,r)=\{x \in X \mid \exists \, a \in A \text{ such that } d(x,a)<r\}$. Define a relation $\lambda$ on the power set of $X$ by
\[A \lambda C \quad  \text{if and only if} \quad \exists \,  r>0 \text{ such that } A\subseteq B(C,r)\text{ and }C\subseteq B(A,r),\]
i.e., the Hausdorff distance between $A$ and $C$ is finite. Then $\lambda$ is an asymptotic resemblance, called the \textbf{metric asymptotic resemblance} or \textbf{asymptotic resemblance induced by the metric $d$.}
\end{example}

\begin{example}\label{definition6}
Let $\mathcal{E}$ be a coarse structure on a set $X.$ For any two subsets $A$ and $B$ of $X,$ define $A \lambda_{\mathcal{E}}B$ if $A \subseteq E[B]$ and $B \subseteq E[A]$ for some $E \in \mathcal{E}.$ Then the relation $\lambda_{\mathcal{E}}$ is an asymptotic resemblance on $X$. We call $\lambda_{\mathcal{E}}$ the \textbf{asymptotic resemblance induced by the coarse structure $\mathcal{E}$}.
\end{example}

For more examples of asymptotic resemblance spaces, the reader is referred to \cite{Honari}.

\begin{definition}\label{definition612}
Let $(X, \lambda)$ be an asymptotic resemblance space. Then $A \subseteq X$ is called (asymptotically) $\textbf{bounded}$ if $A$ is empty or there exists $x \in X$ such that $A\lambda x.$ If $A$ is not (asymptotically) bounded, then we say that $A$ is (asymptotically) \textbf{unbounded.}
\end{definition}

\begin{definition}\label{definition6123}
	Two subsets $A,C$ of an asymptotic resemblance space $(X,\lambda)$ are called \textbf{asymptotically disjoint} if for all asymptotically unbounded subsets $A'\subseteq A$ and $C'\subseteq C,$ one has $A'\bar{\lambda}C'$.
\end{definition}

\begin{proposition}
If $(X, \lambda)$ is an (asymptotically) \textbf{connected asymptotic resemblance space} (i.e., $x \lambda y$ for all $x,y \in X$), then the collection of bounded sets forms a bornology on $X,$ which we call the \textbf{bornology induced by $\lambda$}.
\end{proposition}
\begin{proof}
Straightforward.
\end{proof}

\begin{remark}
When it is clear that the asymptotic resemblance was induced by the coarse structure $\mathcal{E},$ then for the simplicity of notation we will denote $\lambda_{\mathcal{E}}$ by $\lambda.$
\end{remark}

\begin{remark}
In \cite{Honari}, it is shown that if $\lambda$ is the asymptotic resemblance induced by the coarse structure $\mathcal{E},$ then (asymptotically) bounded sets coincide with the (coarsely) bounded sets.
\end{remark}

\begin{remark}
If $\lambda$ is the asymptotic resemblance induced by the coarse structure $\mathcal{E}$ on $X,$ then one can easily show that $X$ is (coarsely) connected if and only if $X$ is (asymptotically) connected.
\end{remark}

Now we will introduce a relation on the power set of a connected coarse space that under certain conditions will induce a coarse proximity. The reader is encouraged to compare the following definition with the definition of the coarse neighborhood operator for large scale spaces, given in \cite{DydakandWeighill}.

\begin{definition}\label{importantdefinition}
	Let $(X,\mathcal{E})$ be a coarse space and let $A,B\subseteq X$ be any two subsets. Define $A\prec B,$ if for every entourage $E\in\mathcal{E},$ we have that $E[A]\subseteq B\cup K$ for some bounded set $K\subseteq X$.
\end{definition}

\begin{remark}\label{importantremark}
The above definition implies that $A \subseteq B$ up to some bounded set $K,$ i.e., $(A\setminus K) \subseteq B.$
\end{remark}

The following proposition introduces equivalent definitions of the $\prec$ relation.

\begin{proposition}\label{equivalentcondition}
Let $(X,\mathcal{E})$ be a connected coarse space, $\lambda$ the asymptotic resemblance induced by $\mathcal{E},$ $\mathcal{B}$ the collection of coarsely bounded sets, and $A,B\subseteq X$ any two subsets. Then the following are equivalent:
\begin{enumerate}[(i)]
\item $A\prec B,$ \label{item1}
\item $A$ and $X\setminus B$ are asymptotically disjoint, \label{item2}
\item For all $E \in \mathcal{E},$ there exists $D \in \mathcal{B}$ such that $\Bigl((A\setminus D) \times ((X\setminus B) \setminus D) \Bigr)\cap E= \emptyset.$ \label{item3}
\end{enumerate}
\end{proposition}

\begin{proof}
((\ref{item1}) $\implies$ (\ref{item2})). Assume $A\prec B.$ For contradiction, assume that $A^{\prime}\subseteq A$ and $C^{\prime}\subseteq (X\setminus B)$ are unbounded subsets such that $A^{\prime}{\lambda}C^{\prime},$ i.e., there exists $E \subseteq \mathcal{E}$ such that $A' \subseteq E[C']$ and $C' \subseteq E[A'].$ Since $A\prec B,$ we have that $E'[A]\subseteq B\cup K$ for some bounded set $K\subseteq X.$ Since $A'\subseteq A,$ we have that 
\[C'\subseteq E[A'] \subseteq E[A] \subseteq B\cup K.\]
Thus, $(C' \setminus K) \subseteq B.$ Since $C'$ is unbounded and $K$ is bounded, $(C' \setminus K)$ is nonempty. But this is a contradiction, since  $(C' \setminus K) \subseteq (X\setminus B),$ by the definition of $C'.$

((\ref{item2}) $\implies$ (\ref{item1})). For contradiction, assume that $A\not\prec B,$ i.e., there exists $E \in \mathcal{E}$ such that $E[A] \not\subseteq B \cup K$ for any bounded $K \subseteq X.$ In other words, $E[A] \cap (X \setminus B)$ is unbounded. Without loss of generality we can assume that $E$ is symmetric. Set $C'=E[A] \cap (X \setminus B).$ For each $c \in C'$ there exists $a \in A$ such that $(c,a) \in E.$ Let $A'$ be the collection of all such $a$'s. Notice that $A'$ is unbounded, since if it is bounded, then so is $E[A'].$ But  $E[A']$ contains $C',$ so it has to be unbounded. So we have an unbounded $A' \subseteq A,$ an unbounded $C' \subseteq (X \setminus B),$ and $E \in \mathcal{E}$ such that
\[C' \subseteq E[A'] \quad \text{and} \quad A' \subseteq E[C'],\]
a contradiction to $A^{\prime}\bar{\lambda}C^{\prime}.$

((\ref{item1}) $\implies$ (\ref{item3})) Let $E \in \mathcal{E}$ be arbitrary. Without loss of generality we can assume that $E$ contains the diagonal. Since $A\prec B,$ there exists $K \in \mathcal{B}$ such that $(E[A] \setminus K) \subseteq B.$ Let $D$ be all those elements of $A$ such that $E[D] \subseteq K.$ Since $K$ is bounded, so is $E[D].$ Since $E$ contains the diagonal, $D$ is bounded as well. Thus, by the construction of $D$ we have that $E[A\setminus D] \subseteq B.$ In other words, if there exists $x \in X$ and $a \in (A \setminus D)$ such that $(x,a) \in E,$ then $x$ cannot be in $(X \setminus B).$ In particular, it cannot be in $((X\setminus B) \setminus D),$ which shows (\ref{item3}).

((\ref{item3}) $\implies$ (\ref{item1})) For contradiction, assume that $A\not\prec B,$ i.e., there exists $E \in \mathcal{E}$ such that $E[A] \cap (X \setminus B)$ is unbounded. Let $D \in \mathcal{B}$ be arbitrary. Then $C:=(E[A] \cap (X \setminus B) ) \setminus (E \cup \triangle)[D]$ is nonempty. Let $c \in C.$ Then there exists $a \in A$ such that $(c,a) \in E.$ What is more, $a \notin D.$ For if $a \in D,$ then $c \in E[D],$ a contradiction. So we have $c \in ((X \setminus B) \setminus D),$ $a \in (A \setminus D)$ and $(a,c) \in E.$ Since $D$ was an arbitrary unbounded subset, this contradicts (\ref{item3}).
\end{proof}

\begin{remark}
Notice that if $(X,\mathcal{E})$ is a connected coarse space, then the collection $\mathcal{B}$ from the above theorem is a bornology.
\end{remark}

Now we introduce a condition under which $\prec$ relation will induce a coarse proximity.

\begin{definition}
	A coarse space $(X,\mathcal{E})$ is called {\bf coarsely normal} if for every pair of subsets $A,B\subseteq X$ such that $A\prec B,$ there is a subset $C\subseteq X$ satisfying $A\prec C\prec B$.
\end{definition}

The reader familiar with \cite{DydakandWeighill} will spot an immediate resemblance to large scale normality defined for large scale structures. Indeed, after translating from large scale structures to coarse structures, the two notions coincide for connected coarse spaces, as the following lemma and proposition show:

\begin{lemma}\label{lemma1}
Let $(X,\mathcal{E})$ be a connected coarse structure, $\mathcal{B}$ the bornology induced by $\mathcal{E},$ and $D_1, D_2 \in \mathcal{B}.$ If $A$ and $B$ are two subsets of $X$ such that $A \prec B$, then the following hold: 
\begin{enumerate}[(i)]
\item $A \cup D_1 \prec B \setminus D_2,$
\item  $A \setminus D_1 \prec B \cup D_2.$
\end{enumerate}
\end{lemma}

\begin{proof}
Straighforward.
\end{proof}

The reader unfamiliar with large scale structures can take (\ref{itemm2}) of the following proposition as the definition of the large scale normality given in \cite{DydakandWeighill}.

\begin{proposition}
Let $(X,\mathcal{E})$ be a connected coarse structure. Then the following are equivalent:
\begin{enumerate}[(i)]
\item $(X,\mathcal{E})$ is coarsely normal, \label{itemm1}
\item For any $A,B \subseteq X,$ define $A \prec^* B$ if $A\subseteq B$ and $A \prec B.$ Then $A\prec^* B$ implies that there exists a subset $C\subseteq X$ satisfying $A\prec^* C\prec^* B$. \label{itemm2}
\end{enumerate}
\end{proposition}

\begin{proof}
To show $(\ref{itemm1}) \implies (\ref{itemm2})$, assume $A\prec^* B,$ i.e., $A \subseteq B$ and $A \prec B$. By coarse normality, this implies the existence of $C' \subseteq X$ such that $A \prec C' \prec B.$ In particular, this shows that there exist bounded sets $D_1$ and $D_2$ such that $A \subseteq C' \cup D_1$ and $C' \subseteq B \cup D_2.$ We can assume that $D_1 \subseteq A$ and $D_2 \subseteq (X \setminus B).$ Set $C=(C' \cup D_1) \setminus D_2.$ By repeated application of Lemma \ref{lemma1}, we have that $A \prec C \prec B.$ Also, $A \subseteq C \subseteq B,$ which follows from the fact that $D_2 \cap A = \emptyset$ (which in particular shows that $D_2 \cap D_1 = \emptyset$). To show $(\ref{itemm2}) \implies (\ref{itemm1}),$ assume $A\prec B.$ In particular, this means that $ \triangle [A]=A \subseteq B \cup D$ for some bounded set $D.$  By Lemma \ref{lemma1}, this means that $A \setminus D \prec^* B,$ and thus there exists $C \subseteq X$ such that $A\setminus D \prec^*C \prec^* B.$  In particular, this means that $A\setminus D \prec C \prec B,$ and by Lemma \ref{lemma1}, we have $A \prec C \prec B.$
\end{proof}

Now we will show that in case of coarse spaces, coarse normality is also equivalent to asymptotic normality, defined in \cite{Honari}. For the convenience of the reader, we recall the definition here.

\begin{definition}
An asymptotic resemblance space $(X,\lambda)$ is \textbf{asymptotically normal} if for all asymptotically disjoint subsets $A_{1},A_{2}\subseteq X$, there are subsets $X_{1},X_{2}\subseteq X$ such that $X=X_{1}\cup X_{2}$, $A_{1}$ is asymptotically disjoint from $X_{1}$, and $A_{2}$ is asymptotically disjoint from $X_{2}$.
\end{definition}

\begin{proposition}
Let $(X,\mathcal{E})$ be a coarse space and $\lambda$ the asymptotic resemblance induced by $\mathcal{E}.$ Then the following are equivalent:
\begin{enumerate}[(i)]
\item $(X,\mathcal{E})$ is coarsely normal, \label{itemitem1}
\item $(X,\lambda)$ is asymptotically normal. \label{itemitem2}
\end{enumerate}
\end{proposition}

\begin{proof}
((\ref{itemitem1}) $\implies$ (\ref{itemitem2})) Assume $A_{1},A_{2}\subseteq X$ such that $A_1$ and $A_2$ are asymptotically disjoint, i.e., $A_1 \prec (X \setminus A_2).$ Thus, there exists $C$ such that $A_1 \prec C \prec (X \setminus A_2).$ Set $X_1= (X \setminus C)$ and $X_2=C.$ Then clearly  $X=X_{1}\cup X_{2}$, $A_{1}$ is asymptotically disjoint from $X_{1}$, and $A_{2}$ is asymptotically disjoint from $X_{2}$.

((\ref{itemitem2}) $\implies$ (\ref{itemitem1})) Assume $A,B \subseteq X$ such that $A \prec B,$ i.e., $A$ and $(X \setminus B)$ are asymptotically disjoint. Thus, there exists $X_{1},X_{2}\subseteq X$ such that $X=X_{1}\cup X_{2}$, $A$ is asymptotically disjoint from $X_{1}$, and $(X \setminus B)$ is asymptotically disjoint from $X_{2}.$ Let $C=X_2.$ Then the following hold:
\begin{enumerate}
\item $A$ is asymptotically disjoint from $X_1=(X \setminus X_2)=(X \setminus C),$
\item $(X \setminus B)$ is asymptotically disjoint from $X_2=C,$
\end{enumerate}
which is the same as saying $A \prec C \prec B.$
\end{proof}

Thanks to the above proposition, it follows from \cite{Honari} that the class of coarsely normal coarse spaces in nonempty. In particular, all metric spaces (with the metric coarse structure) are coarsely normal. Also, notice that in the above proof we used the definition of the $\prec$ relation that involved asymptotic resemblance. In particular, the fact that $\lambda$ was induced by a coarse structure was not used. Therefore, the same proof will show the following proposition:

\begin{proposition}\label{normality_for_asymptotic_resemblance}
Let $(X,\lambda)$ be an asymptotic resemblance space. For any $A,B \subseteq X,$ define $A \prec B$ if and only if $A$ and $X \setminus B$ are asymptotically disjoint. Then the following are equivalent:
\begin{enumerate}[(i)]
\item $A \prec B$ implies that there exists $C \subseteq X$ such that $A \prec C \prec B,$
\item $(X,\lambda)$ is asymptotically normal. \hfill \openbox
\end{enumerate}
\end{proposition}

At this point the reader may be wondering if there exist coarse spaces that are not coarsely normal. Indeed, in \cite{DydakandWeighill} it is shown that there exist such coarse spaces. The following example was inspired by Corollary $11.4$ in that paper (translated to the setting of coarse spaces):
\begin{example}
Let $X=\mathbb{R}^+$ and let $\mathcal{E}'$ be the collection of subsets of $\mathbb{R}^{+} \times \mathbb{R}^{+}$ that consists of finitely many half-lines starting at the $y$ or $x$ axis and parallel to the diagonal. Let $\mathcal{E}$ be the collection of all the subsets of elements of $\mathcal{E'}.$ Then it is easy to see that $\mathcal{E}$ is a coarse structure whose bounded sets are the subsets of $\mathbb{R}^+$ of finite cardinality. Let $A=(0,1)$ and let $B=\mathbb{R}^+ \setminus \mathbb{N}.$ It is clear that $A \prec B.$ Also, notice that any $C\subseteq X$ such that $A \prec C$ needs to contain a set of the form $\mathbb{R}^+ \setminus D,$ where $D$ is a sequence of points diverging to infinity (it is because for any $x \in \mathbb{R}^+$ we can always find $E \in \mathcal{E}$ such that $(0,x) \subseteq E[A]$). However, since we can always draw a half-line parallel to the diagonal that misses countably many points (more precisely, misses all the points in $D \times D$), there exists $E \in \mathcal{E}$ such that $E[C] \supseteq E[\mathbb{R}^+ \setminus D]=\mathbb{R}^+,$ i.e., $E[C]=\mathbb{R}^+.$ But this means that $C \not\prec B$ for any $C$ such that $A \prec C,$ i.e., $(X, \mathcal{E})$ is not coarsely normal.
\end{example}

\section{Coarse Proximities Induced by Coarse Structures and Asymptotic Resemblances}\label{inducing coarse proximities}

Finally, we are ready to prove that $\prec$ relation on a connected coarsely normal space induces a coarse proximity. In the proof, we utilize the characterization of the $\prec$ relation that uses asymptotic resemblance induced by the given coarse structure (see Proposition \ref{equivalentcondition}).

\begin{theorem}\label{inducingproximity}
	Let $(X,\mathcal{E})$ be a connected coarse space and $\mathcal{B}$ the bornology induced by $\mathcal{E}$. The relation $\prec$ induces a coarse proximity on the pair $(X,\mathcal{B})$ if and only if $(X,\mathcal{E})$ is coarsely normal. 
\end{theorem}

\begin{proof}
 If $\prec$ induces a coarse proximity on the pair $(X,\mathcal{B}),$ then $(X,\mathcal{E})$ is coarsely normal by (\ref{intermediate}) of Theorem \ref{propertiesofcoarseneighborhoods}. To prove the converse, assume that $(X,\mathcal{E})$ is coarsely normal. To show that $\prec$ induces a coarse proximity, it is enough to show that the relation $\prec$ satisfies $(1)$ through $(6)$ of Theorem \ref{propertiesofcoarseneighborhoods}. To show (\ref{boundeddontmatter}), let $D \in \mathcal{B}$ be arbitrary. Since subsets of bounded sets are bounded, there is no such $D' \subseteq D$ such that $D'$ is unbounded. Therefore, $X \bar{\lambda} D$ is satisfied vacuously, i.e., $X \prec (X \setminus D).$ To show (\ref{containment}), assume $A \prec B.$  For contradiction, assume that  $C:=A \cap (X \setminus B)$ is unbounded. Then $C \subseteq A,$ $C\subseteq (X \setminus B),$ and $C$ is unbounded. By Proposition $2.22$ of \cite{Honari}, we have that $C \lambda C,$ which contradicts the fact that $A$ is asymptotically disjoint from $(X \setminus B).$ Thus, it has to be that $A \cap (X \setminus B)$ is bounded, i.e., $A$ is contained in $B$ up to some bounded set. To show (\ref{subsetandsuperset}), assume that $A \subseteq B \prec C \subseteq D.$ If $A \not\prec D,$ then there exist unbounded $A' \subseteq A \subseteq B$ and unbounded $D' \subseteq (X \setminus D) \subseteq (X \setminus C)$ such that $A' \lambda D',$ a contradiction to $ B \prec C.$ So it has to be that $A \prec D.$ To show (\ref{intersection}), assume $A \prec B_1$ and $A \prec B_2,$ i.e., $A$ is asymptotically disjoint from $(X \setminus B_1)$ and $(X \setminus B_2).$ For contradiction, assume that $A\not\prec (B_1 \cap B_2),$ i.e., there exists unbounded $A' \subseteq A$ and unbounded $C' \subseteq X \setminus (B_1 \cap B_2), $ such that $A' \lambda C'.$ However, notice that $X \setminus (B_1 \cap B_2) = (X \setminus B_1) \cup (X \setminus B_2).$ Thus, there has to exist unbounded $C'' \subseteq C'$ such that $C'' \subseteq (X \setminus B_1)$ or $C'' \subseteq (X \setminus B_2)$ (otherwise $C'$ would be bounded, being the union of two bounded sets). Without loss of generality assume that $C'' \subseteq (X \setminus B_1).$ Notice that since $A' \lambda C',$ by Proposition 2.6 of \cite{Honari}, there exists $A'' \subseteq A'$ such that $A'' \lambda C''.$ Clearly $A''$ has to be unbounded (for if it is bounded, then there exists $x \in X$ such that $x \lambda A'' \lambda C'',$ contradicting the fact that $C''$ is unbounded). So we have unbounded $A'' \subseteq A,$ and unbounded $C'' \subseteq (X \setminus B_1)$ such that $A'' \lambda C'',$ a contradiction to $A \prec B_1.$ So it has to be the case that $A\prec (B_1 \cap B_2).$ To show the converse, assume $A\prec (B_1 \cap B_2).$ If without loss of generality $A \not\prec B_1,$ then there exist unbounded $A' \subseteq A$ and unbounded $C' \subseteq (X \setminus (B_1))\subseteq (X \setminus (B_1 \cap B_2))$ such that $A' \lambda C',$ a contradiction to $A\prec (B_1 \cap B_2).$ To show (\ref{neighborhoodofcomplement}), assume $A \prec B$ and for contradiction assume that $(X\setminus B) \not\prec (X \setminus A).$ Then there exist unbounded $B' \subseteq (X \setminus B)$ and unbounded $A' \subseteq (X \setminus (X \setminus A))=A$ such that $B' \lambda A',$ which contradicts $A \prec B.$ The converse is shown similarly. Finally, (\ref{intermediate}) is the coarse normality.
\end{proof}

\begin{corollary}
Let $(X,\mathcal{E})$ be a connected coarsely normal coarse space, $\lambda$ the asymptotic resemblance induced by $\mathcal{E},$ $\mathcal{B}$ the bornology induced by $\mathcal{E},$ and $A,B\subseteq X$ any two subsets. Define the relation ${\bf b}$ on the power set of $X$ by any of the following equivalent conditions:
\begin{enumerate}[(i)]
\item $A {\bf b} B$ if and only if there exists $E \in \mathcal{E}$ such that $E[A] \cap B$ is unbounded,
\item $A {\bf b} B$ if and only if there exists an unbounded $A' \subseteq A$ and an unbounded $B' \subseteq B$ such that $A' \lambda B',$
\item $A {\bf b} B$ if and only if there exists $E \in \mathcal{E}$ such that for all $D \in \mathcal{B},$ 
\[\Bigl((A\setminus D) \times (B \setminus D) \Bigr)\cap E \neq \emptyset.\]
\end{enumerate}
Then ${\bf b}$ is a coarse proximity.
\end{corollary}
\begin{proof}
This is a direct consequence of Theorem \ref{inducingproximity}, Proposition \ref{equivalentcondition}, and Theorem \ref{<<inducesproximity}.
\end{proof}

\begin{remark}
Notice that $(iii)$ of the above corollary is in line with the definition of the metric coarse proximity given in \cite{Coarseproximity}, where two subsets $A$ and $B$ of a metric space $(X,d)$ are coarsely close if and only if there exists $\epsilon < \infty$ such that for all bounded sets $D$, there exists $a \in (A \setminus D)$ and $b \in (B \setminus D)$ such that $d(a,b) < \epsilon.$
\end{remark}

Since in the proof of Theorem \ref{inducingproximity} we have used the characterization of the $\prec$ relation that uses the induced asymptotic resemblance, the proof of that theorem also shows that connected asymptotically normal asymptotic resemblance spaces naturally induce coarse proximities, as in the following corollary.

\begin{corollary}\label{asymptotic_resemblance_corollary}
	Let $(X,\lambda)$ be a connected asymptotic resemblance space and $\mathcal{B}$ the bornology induced by $\lambda$. For any $A,B \subseteq X,$ define $A \prec B$ if and only if $A$ and $X \setminus B$ are asymptotically disjoint. The relation $\prec$ induces a coarse proximity on the pair $(X,\mathcal{B})$ if and only if $(X,\mathcal{E})$ is asymptotically normal. 
\end{corollary}

\begin{proof}
This is a direct consequence of Proposition \ref{normality_for_asymptotic_resemblance} and the proof of Theorem \ref{inducingproximity}.
\end{proof}

\begin{corollary}
Let $(X,\lambda)$ be a connected asymptotically normal asymptotic resemblance space, $\mathcal{B}$ the bornology induced by $\lambda,$ and $A,B\subseteq X$ any two subsets. Define the relation ${\bf b}$ on the power set of $X$ by 
\vspace{2mm}

$A {\bf b} B$ if and only if there exists an unbounded $A' \subseteq A$ and an unbounded $B' \subseteq B$ such that $A' \lambda B',$
\vspace{2mm}

\noindent i.e., $A$ and $B$ are not asymptotically disjoint. Then ${\bf b}$ is a coarse proximity.
\end{corollary}

\begin{proof}
This is a direct consequence of Corollary \ref{asymptotic_resemblance_corollary} and Theorem \ref{<<inducesproximity}.
\end{proof}

\begin{question}
Is every coarse proximity space induced by some coarse structure / asymptotic resemblance structure?
\end{question}

\bibliographystyle{abbrv}
\bibliography{An_alternative_description_of_coarse_proximities}{}

\end{document}